\documentclass{article}
\usepackage{amsmath}
\usepackage{amssymb}
\usepackage{amsfonts,,mathdots}
\usepackage{latexsym,longtable}

\usepackage{tikz}
\usetikzlibrary{matrix}
\begin{document}
\newcommand{\B}{{\cal B}}
\newcommand{\D}{{\cal D}}
\newcommand{\E}{{\cal E}}
\newcommand{\F}{{\cal F}}
\newcommand{\A}{{\cal A}}
\newcommand{\Hh}{{\cal H}}
\newcommand{\Pp}{{\cal P}}
\newcommand{\Z}{{\bf Z}}
\newcommand{\T}{{\cal T}}
\newcommand{\ZZ}{{\mathbb{Z}}}
\newcommand{\qed}{\hphantom{.}\hfill $\Box$\medbreak}
\newcommand{\proof}{\noindent{\bf Proof \ }}
\renewcommand{\theequation}{\thesection.\arabic{equation}}
\newtheorem{theorem}{Theorem}[section]
\newtheorem{lemma}[theorem]{Lemma}
\newtheorem{corollary}[theorem]{Corollary}
\newtheorem{remark}[theorem]{Remark}
\newtheorem{example}[theorem]{Example}
\newtheorem{definition}[theorem]{Definition}
\newtheorem{construction}[theorem]{Construction}


\medskip
\title{Constructions of Augmented Orthogonal Arrays \thanks{Research supported by NSFC grants 11431003 (L. Ji), 11301370 (M. Liang).
 }}

  \author{{\small   Xin Wang$^1$,\ Lijun Ji$^1$\footnote{Corresponding author,  E-mail: jilijun@suda.edu.cn.},\ Yun Li$^1$ \  and Miao Liang$^2$} \\
 {\small $^1$ Department of Mathematics, Soochow University, Suzhou
 215006, China}\\
{\small E-mail: jilijun@suda.edu.cn}\\
{\small $^2$ Department of Mathematics and Physics, Suzhou Vocational University, Suzhou 215104, P.R.China.}\\
}

\date{}
\maketitle
\begin{abstract}
\noindent \\ Augmented orthogonal arrays (AOAs) were introduced by Stinson, who showed the equivalence between  ideal ramp
schemes and augmented orthogonal arrays (Discrete Math. 341 (2018),  299-307).
In this paper, we show that there is  an AOA$(s,t,k,v)$ if and only if there is an OA$(t,k,v)$ which can be partitioned into $v^{t-s}$ subarrays, each being an OA$(s,k,v)$, and that there is a linear AOA$(s,t,k,q)$ if and only if there is a linear maximum distance separable (MDS) code of length $k$ and dimension $t$ over $\mathbb{F}_q$ which contains a linear MDS subcode of length $k$ and dimension $s$ over $\mathbb{F}_q$.  Some constructions for AOAs and some new infinite classes of AOAs are also given.

\medskip


\medskip


\end{abstract}


\section{Introduction}

An {\em orthogonal array}, denoted by OA$(t,k,v)$, is a $v^t$ by $k$ array with entries from a symbol set $X$ of size $v$ such
that each of its $v^t\times t$ subarrays contains every $t$-tuple from $X^t$
exactly once. It is well known that the existence of an OA$(2,k,v)$  is equivalent to the existence of  a set of $k-2$ mutually
orthogonal Latin squares (MOLSs) of side $v$.
The following elegant theorem was due to Bush \cite{Bush1952a}.

\begin{theorem}
[\cite{Bush1952a}]
\label{OA-primepower}
For any prime power $q$ and any positive integer $t$ with
$t\leq q$, an OA$(t,q+1,q)$ exists, and if $q\geq 4$ is a
power of $2$ then  an OA$(3,q+2,q)$ exists.
\end{theorem}

Orthogonal arrays belong to an important and high-profile area of
combinatorics and statistics. They are of fundamental importance as
ingredients in the construction of other useful combinatorial
objects (see \cite{BJL1999, CD2007,HSS1999}).

Recently, Stinson introduced a concept of
augmented orthogonal arrays in order to characterize  ideal ramp
schemes \cite{Stinson2017}.

An augmented orthogonal array, denoted by AOA$(s, t, k, v)$, is a $v^t$ by $k + 1$
array $A$ that satisfies the following properties:

1. the first $k$ columns of $A$ form an orthogonal array OA$(t, k, v)$ on a symbol set $X$ of
size $v$; \\
\indent 2. the last column of $A$ contains symbols from a set $Y$ of size $v^{t-s}$; \\
\indent 3. any $s$ of the first $k$ columns of $A$, together with the last column of $A$, contain all
possible $(s + 1)$-tuples from $X^s\times Y$ exactly once.

Informally, an $(s,t,n)$ ramp scheme \cite{BM1985} is a method of distributing secret information (called {\em shares}) to $n$ players, in
such a way  that any $t$ of the players can compute a predetermined {\em secret},  no subset of $s$ players can determine the secret.  The parameters of a ramp scheme satisfy the conditions $0\leq s<t\leq n$.
A $(t-1,t,n)$ ramp scheme is usually called a $(t,n)$ threshold scheme \cite{Shamir1979}. If there are $v$ possible shares in an $(s,t,n)$ ramp scheme, then the number of possible secrets is bounded above by $v^{t-s}$. If an $(s,t,n)$ ramp scheme can be constructed with  $v^{t-s}$ possible secrets (given $v$ possible shares), then the ramp scheme is {\em ideal}
\cite{Stinson2017}.
Stinson showed the equivalence between ideal ramp schemes and augmented orthogonal arrays \cite{Stinson2017}.

\begin{theorem}
[\cite{Stinson2017}]
There is an ideal $(s,t,n)$ ramp scheme defined over a set of $v$ shares if and only if there is an AOA$(s,t,n,v)$.
\end{theorem}

For more information on ideal ramp schemes, we refer the reader to \cite{Stinson2017} and the references therein.

An OA$(t,k,v)$, say $A$, over $X$ is called $s$-resolvable if the array $A$ can be partitioned into $v^{t-s}$ subarrays $P_i$, $1\leq i\leq v^{t-s}$, such that each $P_i$ is an OA$(s,k,v)$ over $X$. Usually, when $s=1$, we simply call it resolvable.  We give a characterization of AOAs in terms of $s$-resolvable OAs as follows.

\begin{theorem}
\label{equivalence}
There is an AOA$(s,t,k,v)$ if and only if there is an $s$-resolvable OA$(t,k,v)$.
\end{theorem}

\proof Let $A$ be an AOA$(s,t,k,v)$ where the elements in the first $k$ columns come from the symbol set $X$ and the symbols in the last column come from the set $Y$ of size $v^{t-s}$. Take an arbitrary element $y$ from $Y$ and consider the subarray $A_y$ consisting of all rows with the last coordinate being $y$. Since any $s$ of the first $k$ columns of $A$, together with the last column of $A$, contain all
possible $(s + 1)$-tuples from $X^s\times Y$ exactly once, each $s$-tuple from $X^s$ occurs exactly once in any $s$ of the first $k$ columns of $A_y$. It follows that the array $A_y'$ obtained by deleting the last column of $A_y$ is an OA$(s,k,v)$ over $X$.  Since the first $k$ columns of $A$ form an OA$(t, k, v)$ and
all $A_y'$ form a partition of the first $k$ columns of $A$, the first $k$ columns of $A$ form an $s$-resolvable OA$(t,k,v)$ over $X$.

Conversely, let $A$ be an $s$-resolvable OA$(t,k,v)$ over $X$. Then $A$ can be partitioned into $v^{t-s}$ subarrays $A_i$, $1\leq i\leq v^{t-s}$, such that each $A_i$ is an OA$(s,k,v)$ over $X$. For $1\leq i\leq v^{t-s}$, append a column vector of dimension $v^s$ with all coordinates being $i$ to $A_i$ to obtain a new array $A_i'$. It is routine to check that the
new array consisting of $A_i'$, $1\leq i\leq v^{t-s}$, is an   AOA$(s,t,k,v)$. \qed

By Theorem \ref{equivalence},
we only need to pay attention to constructions of $s$-resolvable OAs.  Clearly,  deleting $i$ columns  from an $s$-resolvable OA$(t,k,v)$ gives an $s$-resolvable OA$(t,k-i,v)$ for $1\leq i\leq k-t$.
Note that  $s$-resolvable OAs are of interest in design theory.  Similar combinatorial structures such as  resolvable $t$-designs and 2-resolvable Steiner quadruple systems have been widely studied, see \cite{CD2007,HP1992}.

Let $q$ be a prime power and $\mathbb{F}_q$ the finite field of order $q$. An OA$(t,k,q)$ over $\mathbb{F}_q$ is called  linear if the set of row vectors can be viewed as a subspace of dimension $t$ of a vector space of dimension $k$ over $\mathbb{F}_q$.
Stinson gave some constructions for AOAs and linear AOAs, which are listed in the following.

\begin{theorem}
[\cite{Stinson2017}]
\label{AOA(t-1,t,k,v)}
An AOA$(t-1, t, k, v)$ is equivalent to an OA$(t, k+1, v)$.
\end{theorem}

\begin{theorem}
[\cite{Stinson2017}]
\label{OA-AOA}
If there exists an OA$(t, k + t-s, v)$, then there exists an AOA$(s, t, k, v)$.
\end{theorem}

\begin{theorem}
[\cite{Stinson2017}]
Suppose there is a linear OA$(t-s, t, q)$. Then there exists an AOA$(s, t, t, q)$.
\end{theorem}

\begin{theorem}
[\cite{Stinson2017}]
\label{primepower}
 Suppose $q$ is a prime power and $1 \leq s < t \leq k \leq q$. Then there exists a
linear AOA$(s, t, k, q)$.
\end{theorem}

\begin{theorem}
[\cite{Stinson2017}]
Suppose $q$ is a prime power and $s \leq q-1$. Then there exists an AOA$(s, q+
1, q + 1, q)$ but there does not exist an OA$(q + 1, 2(q + 1)-s, q)$.
\end{theorem}

In \cite{Stinson2017}, Stinson mentioned  three problems for possible future study:
\begin{itemize}
\item[1.]  It is interesting to give constructions of linear AOA$(s, t , k, q)$'s but the corresponding OA$(t , k+t-s,q)$'s (linear or not) do not exist.
\item[2.]  A related problem is to find parameter sets for which linear AOAs exist but linear OAs do not exist.
\item[3.]  A third problem concerns constructions over alphabets of non-prime power order. Constructions of ideal ramp schemes over alphabets of non-prime power order would also be of interest. Again, the most interesting cases are those where an AOA exists but the corresponding OA does not exist.
\end{itemize}

By Theorem \ref{AOA(t-1,t,k,v)},  an AOA$(1,2,k,v)$ is equivalent to an OA$(2,k+1,v)$ and a resolvable OA$(2,k,v)$. A lot of  work has been done on the existence of OA$(2,k,v)$, see \cite{CD2007}.  In this paper, we establish some constructions of AOAs  with $t\geq 3$.
Many infinite classes of AOAs over alphabets of non-prime power orders and linear AOAs are obtained. We also identify some parameters  for which there exists a linear AOA$(s, t , k, q)$ but there does not exist an OA$(t , k+t-s,q)$.

The remainder of this paper is organized as follows. In Section 2, we give some constructions of AOAs over alphabets of non-prime power order. In Section 3, we rewrite Stinson's construction for linear AOAs and give  some direct constructions of linear AOAs.  Finally, we identify some parameters  for which there exists a linear AOA$(s, t , k, q)$ but there does not exist an OA$(t , k+t-s,q)$ in Section 4.


\section{Constructions of AOAs of non-prime power order}

In this section, we give some constructions of AOAs over alphabets of non-prime power order.

\begin{theorem}
\label{product-AOA}
If there is an AOA$(s,t,k,v)$ and an AOA$(s,t,k,u)$,  there is an AOA$(s,t,k,uv)$.
\end{theorem}

\proof By assumption and the equivalence between AOAs and $s$-resolvable OAs in Theorem \ref{equivalence}, we can assume that $A$ is an $s$-resolvable OA$(t,k,v)$ over $G$ and $B$ is an $s$-resolvable OA$(t,k,u)$ over $G'$. Then $A$ can be partitioned into $v^{t-s}$ subarrays $A_i$, $1\leq i\leq v^{t-s}$, each being an OA$(s,k,v)$ over $G$, and $B$ can be partitioned into $v^{t-s}$ subarrays $B_j$, $1\leq j\leq u^{t-s}$, each being an OA$(s,k,u)$ over $G'$. For each row vector $(a_{i,\ell,1},a_{i,\ell,2},\ldots,a_{i,\ell,k})$ of $A_i$ and each row  vector $(b_{j,\ell',1},b_{j,\ell',2},\ldots,b_{j,\ell',k})$ of $B_j$ where $1\leq \ell \leq v^s$ and $1\leq \ell'\leq u^s$, construct a row vector
 $$((a_{i,\ell,1},b_{j,\ell',1}),(a_{i,\ell,2},b_{j,\ell',2}),\ldots,(a_{i,\ell,k},b_{j,\ell',k})).$$
By the well known product construction for orthogonal arrays (for example see \cite{Bush1952b}), the subarray $C_{i,j}$ consisting of all row vectors constructed from $A_i$ and $B_j$ is an OA$(s,k,uv)$ over $G\times G'$, and all these $(uv)^{t-s}$ subarrays $C_{i,j}$ form an OA$(t,k,uv)$. Therefore the conclusion holds by Theorem \ref{equivalence}. \qed

Applying Theorem \ref{product-AOA} with the known AOAs can yield many AOAs of non-prime power orders.

\begin{theorem}
 Suppose each prime divisor of $v$ is not less than $k$  and $1 \leq s < t \leq k$. Then there exists an
AOA$(s, t, k, v)$.
\end{theorem}

\proof Write $v=p_1p_2\cdots p_r$ where $p_1,\ldots,p_r$ are  primes. Since $p_i\geq k$ by assumption, there is an AOA$(s,t,k,p_i)$ by Theorem  \ref{primepower}.
Applying Theorem \ref{product-AOA}  gives the result. \qed

\begin{theorem}
\label{AOA(1,k-1,k,v)}
Suppose that $v$ and $k$ are positive integers, $v\not\equiv 2\pmod 4$ and $k\geq 3$. There is an AOA$(1,k-1,k,v)$.
\end{theorem}

\proof Write $v=2^{\alpha}p_1p_2\cdots p_r$ where $p_1,\ldots,p_r$ are odd primes. Then $\alpha\neq 1$ by assumption.
When $\alpha\geq 2$, by Theorem \ref{AOA(1,k-1,k,q)} there is a linear AOA$(1,k-1,k,2^{\alpha})$ over $\mathbb{F}_{2^{\alpha}}$ and  a linear AOA$(1,k-1,k,p_i)$ over $\mathbb{F}_{p_i}$ for $1\leq i\leq r$.
Applying Theorem \ref{product-AOA}  gives an AOA$(1,k-1,k,v)$ over $\mathbb{F}_{2^{\alpha}}\times \mathbb{F}_{p_1}\times \cdots \times \mathbb{F}_{p_r}$. When $\alpha=0$, the result  is obtained similarly. \qed

\begin{theorem}
\label{zero-sum} Let $v,k$ be  positive integers with  $k\equiv 0\pmod 2$. Then there is an AOA$(1,k-1,k,v)$.
\end{theorem}

\proof Let $A$ consist of the following $k$-tuples:
$$(x_1,x_2,\ldots,x_k),$$
where $x_1,x_2,\ldots,x_k\in \mathbb{Z}_v$ and $x_1+x_2+\cdots+x_k\equiv 0\pmod v$. By the well known zero-sum construction for OAs, the array $A$ is an OA$(k-1,k,v)$.
By Theorem \ref{equivalence}, we only need to show that this OA is resolvable.

For $a_1,\ldots,a_{k-2}\in \mathbb{Z}_v$, let $A_{a_1,\ldots,a_{k-2}}$ consist of the following $v$ $k$-tuples:
$$(a_1+b,a_2-b,\ldots,a_{k-3}+b,a_{k-2}-b,b,-(a_1+\cdots+a_{k-2})-b),$$
where $b\in \mathbb{Z}_v$. Since the $i$-th coordinate runs through $\mathbb{Z}_v$ when $b$ runs through $\mathbb{Z}_v$,
$A_{a_1,\ldots,a_{k-2}}$ is an OA$(1,k,v)$, thereby
this OA is resolvable.  \qed

\begin{theorem}
\label{OA(s,t,v)-AOA(s,t,t,v)}
If there is an OA$(s,t,v)$ with $s<t$ then there is an AOA$(s,t,t,v)$.
\end{theorem}

\proof Let $A=(a_{i,j})$ be an OA$(s,t,v)$ over $\mathbb{Z}_v$. For any $(t-s)$-tuple $(b_1,b_2,\ldots,b_{t-s})$ from $\mathbb{Z}_v^{t-s}$, let $A_{b_1,b_2,\ldots,b_{t-s}}$ consist of row vectors  $(a_{i,1}+b_1,a_{i,2}+b_2,\ldots,a_{i,t-s}+b_{t-s},a_{i,t-s+1},\ldots,a_{i,t})$, $1\leq i\leq v^s$.
Clearly,   $A_{b_1,b_2,\ldots,b_{t-s}}$ is an  OA$(s,t,v)$ over $\mathbb{Z}_v$. It is routine to check that all $A_{b_1,b_2,\ldots,b_{t-s}}$ form an OA$(t,t,v)$.
So, an $s$-resolvable OA$(t,t,v)$ exists and the conclusion follows from Theorem \ref{equivalence}. \qed

\begin{corollary}
[\cite{Stinson2017}] Suppose $q$ is a prime power and $1 \leq s \leq k\leq q+1$. Then there is an AOA$(s,k,k,q)$.
\end{corollary}

\proof Since $q$ is a prime power, there is an OA$(s,k,q)$ by Theorem \ref{OA-primepower}. The conclusion then follows by Theorem \ref{OA(s,t,v)-AOA(s,t,t,v)}. \qed

Let $G$ be an abelian group of order $n$.
An  $(n,k,1)$-difference matrix (DM) is an $n\times k$ matrix $D=(d_{i,j})$ with entries from $G$ such that for $1\leq j<\ell\leq k$, the difference list
$$\{a_{i,j}-a_{i,\ell}\colon 1\leq i\leq n\}$$
contains  each element of $G$ exactly once.  It is well known that an  $(n,k,1)$-DM can be used to construct an OA$(2,k+1,n)$ (for example, see \cite{CD2007}).
It is proved that  there is an OA$(3,5,n)$ if there is an $(n,4,1)$-DM \cite{JY2010}.

\begin{lemma}
\label{AOA(1,3,5,odd)}
There is an AOA$(1,3,5,n)$ for any odd integer $n\geq 4$.
\end{lemma}

\proof From \cite{Ge2005}, there is an $(n,4,1)$-DM $D=(d_{i,j})$ over an abelian group $G$ of order $n$. From \cite{JY2010}, the array consisting of the following vectors
$$(d_{i,1}+u,d_{i,2}+u,d_{i,3}+u+e,d_{i,4}+u+e,e)$$
where $1\leq i\leq n$ and $u,e\in G$, is an OA$(3,5,n)$ over $G$.  By Theorem \ref{equivalence}, we only need to show that this OA is resolvable.

Clearly, for $1\leq i\leq n$ and $e'\in G$, the subarray $A_{i,e'}$ consisting of the vectors $(d_{i,1}+u,d_{i,2}+u,d_{i,3}+u+u+e',d_{i,4}+u+u+e',u+e')$, $u\in G$, is an OA$(1,5,n)$. Hence, this OA is resolvable. \qed

Ji and Yin introduced a concept of $(n,4,1)$-DM with an adder in order to construct an OA$(3,6,n)$ \cite{JY2010}.

Let $D=(d_{ij})$ be an $(n,4,1)$-DM over
 an abelian group $G$.
An $n$-tuple $s=(s_1,s_2,\ldots,s_n)^T$ over $G$ is called an {\em
adder} of the difference matrix $D$ if $\{s_1,s_2,\ldots,s_n\}=G$ and
the matrix
$$D^s=(d'_{ij}),\ {\rm where}\ d'_{ij}=d_{ij}\ {\rm for}\ j\in
\{1,2\}\ {\rm and}\ d'_{ij}=d_{ij}+s_i\ {\rm for}\ j\in \{3,4\},$$
is also an $(n,4,1)$-DM over the group $G$.

\begin{lemma}
\label{AOA(1,3,6,12)}
There is an AOA$(1,3,6,n)$ for $n\in \{12,24\}$.
\end{lemma}

\proof From \cite{JY2010}, there is a $(12,4,1)$-DM $D$ with an adder $s$ over $\mathbb{Z}_6\times \mathbb{Z}_2$, and the array consisting of the following row vectors
$$(d_{i,1}+u,d_{i,2}+u,d_{i,3}+u+e+s_i,d_{i,4}+u+e+s_i,e,e+s_i)$$
where $1\leq i\leq 12$ and $u,e\in \mathbb{Z}_6\times \mathbb{Z}_2$, is an OA$(3,6,12)$.  By Theorem \ref{equivalence}, we only need to show that this OA is resolvable.

A mapping $\sigma: \mathbb{Z}_6\times \mathbb{Z}_2\rightarrow \mathbb{Z}_6\times \mathbb{Z}_2$ is defined as follows:
\[
\sigma((x,y))=\left \{ \begin{array}{ll}
(x,y) & {\rm if}\ (x,y)\in \{0,2,4\}\times \{0\}, \\
(x+1,y) & {\rm if}\ (x,y)\in \{0,2,4\}\times \{1\}, \\
(x+1,y+1) & {\rm if}\ (x,y)\in \{1,3,5\}\times \{0\}, \\
(x,y+1) & {\rm if}\ (x,y)\in \{1,3,5\}\times \{1\}. \\
\end{array}
\right .
\]
It is easy to see that $\{u+\sigma(u): u\in \mathbb{Z}_6\times \mathbb{Z}_2\}=\mathbb{Z}_6\times \mathbb{Z}_2$.
Clearly, for $1\leq i\leq 12$ and $e'\in \mathbb{Z}_6\times \mathbb{Z}_2$, the subarray $A_{i,e'}$ consisting of the row vectors $(d_{i,1}+u,d_{i,2}+u,d_{i,3}+u+\sigma(u)+e'+s_i,d_{i,4}+u+\sigma(u)+e'+s_i,\sigma(u)+e',\sigma(u)+e'+s_i)$, $u\in \mathbb{Z}_6\times \mathbb{Z}_2$, is an OA$(1,6,12)$. Hence, this OA$(3,6,12)$ is resolvable.

From \cite{JY2010}, there is a $(24,4,1)$-DM $D$ with an adder $s$  over $\mathbb{Z}_3\times \mathbb{F}_8$, and the array consisting of the following row vectors
$$(d_{i,1}+u,d_{i,2}+u,d_{i,3}+u+e+s_i,d_{i,4}+u+e+s_i,e,e+s_i)$$
where $1\leq i\leq 24$ and $u,e\in \mathbb{Z}_3\times \mathbb{F}_8$, is an OA$(3,6,24)$.  By Theorem \ref{equivalence}, we only need to show that this OA is resolvable.

Let $\alpha$ be a primitive element of $\mathbb{F}_8$ satisfying $1+\alpha+\alpha^3=0$. Define a mapping $\sigma: \mathbb{Z}_3\times \mathbb{F}_8\rightarrow \mathbb{Z}_3\times \mathbb{F}_8$  by $\sigma((x,y))=(x,\alpha y)$ for $(x,y)\in \mathbb{Z}_3\times \mathbb{F}_8$. It is easy to see that $\{(x,y)+(x,\alpha y)\colon (x,y)\in \mathbb{Z}_3\times \mathbb{F}_8\}=\mathbb{Z}_3\times \mathbb{F}_8$. Then for $1\leq i\leq 24$ and $e'\in \mathbb{Z}_3\times \mathbb{F}_8$, the subarray $A_{i,e'}$ consisting of the row vectors $(d_{i,1}+u,d_{i,2}+u,d_{i,3}+u+\sigma(u)+e'+s_i,d_{i,4}+u+\sigma(u)+e'+s_i,\sigma(u)+e',\sigma(u)+e'+s_i)$, $u\in \mathbb{Z}_3\times \mathbb{F}_8$, is an OA$(1,6,24)$. Hence, this OA$(3,6,24)$ is also resolvable.  \qed

\begin{theorem}
\label{AOA(1,3,6,v)}
Let $v$ be a positive integer satisfying
gcd$(v,4)\neq 2$ and gcd$(v,18)\neq 3$. Then there is an
AOA$(1,3,6,v)$.
\end{theorem}

\proof The result for $v\in \{12,24\}$ follows from
Lemma \ref{AOA(1,3,6,12)}. For other values of $v$,  write
$v=2^{\alpha}3^{\beta}p_1^{\gamma_1}p_2^{\gamma_1}\cdots
p_r^{\gamma_r}$ for its prime factorization, where $p_j\geq 5$. By
assumption, we know that $\alpha\neq 1$ and $(\alpha,\beta)\neq
(0,1)$. For $\beta\neq 1$, since there is a linear AOA$(1,3,3^{\beta}+1,3^{\beta})$ by Theorem \ref{AOAt-sGeq2}, deleting its $3^{\beta}-5$ columns  yields an AOA$(1,3,6,3^{\beta})$. Similarly, we can obtain an AOA$(1,3,6,p_j^{\gamma_j})$ from Theorem \ref{AOAt-sGeq2} and an AOA$(1,3,6,2^{\alpha})$ from Theorem \ref{AOA(1,3,q+2,q)}.
Applying Theorem \ref{product-AOA} gives an
AOA$(1,3,6,v)$. If $\beta=1$, then $\alpha\geq 2$. When $\alpha$ is
even, applying Theorem \ref{product-AOA} with the known AOA$(1,3,6,12)$ and
AOA$(1,3,6,2^{\alpha-2}p_1^{\gamma_1}p_2^{\gamma_1}\cdots
p_r^{\gamma_r})$ gives an AOA$(1,3,6,v)$.  When $\alpha$ is odd, applying Theorem
\ref{product-AOA} with the known AOA$(1,3,6,24)$
and AOA$(1,3,6,2^{\alpha-3}p_1^{\gamma_1}p_2^{\gamma_1}\cdots
p_r^{\gamma_r})$ gives an AOA$(1,3,6,v)$. \qed

\begin{lemma}
\label{AOA(1,3,6,15)}
There is an AOA$(1,3,6,v)$ for  $v\in \{15,21\}$.
\end{lemma}

\proof From \cite{JY2010}, there is a $(v,4,1)$-DM $D$ with an adder $s$ over $\mathbb{Z}_{v}$, and the array consisting of the following vectors
$$(d_{i,1}+u,d_{i,2}+u,d_{i,3}+u+e+s_i,d_{i,4}+u+e+s_i,e,e+s_i)$$
where $1\leq i\leq v$ and $u,e\in \mathbb{Z}_{v}$, is an OA$(3,6,v)$.  By Theorem \ref{equivalence}, we only need to show that this OA is resolvable.

Clearly, for $1\leq i\leq v$ and $e'\in \mathbb{Z}_{v}$, the subarray $A_{i,e'}$ consisting of the vectors $(d_{i,1}+u,d_{i,2}+u,d_{i,3}+u+u+e',d_{i,4}+u+u+e',u+e',u+e'+s_i)$, $u\in \mathbb{Z}_{v}$, is an OA$(1,6,v)$. Hence, this OA is resolvable. \qed

Applying Theorem \ref{product-AOA} with the known AOAs in Lemma \ref{AOA(1,3,5,odd)} and Theorem \ref{AOA(1,3,6,v)} yields the following.

\begin{theorem}
\label{AOA(1,3,5,v)}
Let $v\geq 4$ be an integer. If $v\not\equiv 2 \pmod 4$, then an AOA$(1,3, 5, v)$ exists.
\end{theorem}

 Blanchard proved that for  any positive integers $t$ and $k$, $t \leq k$, there is an integer $v^* = v^*(t, k)$ such that for any integer $v\geq v^*$ there is an OA$(t, k, v)$ \cite{Blanchard1994}. Such a fact was pointed by Moh\'{a}csy in \cite{Mohacsy2013}. By  Blanchard's asymptotic existence result  and Theorem \ref{OA-AOA}, asymptotic existence result for AOAs also holds.

 \begin{theorem}
For any positive integers $s,t,k$ with $s< t\leq k$, there is an integer  $v^* = v^*(t, k+t-s)$ such that there is an AOA$(s,t, k, v)$ for any integer $v\geq v^*$.
\end{theorem}


\section{Linear AOAs}

Stinson \cite{Stinson2017} presented an effective construction of linear AOAs as follows.

\begin{construction}
[\cite{Stinson2017}]
\label{linear AOA}
 Suppose that $q$ is a prime power. Suppose that there is a $t$ by $k + t-s$
matrix $M$, having entries from the finite field $\mathbb{F}_q$ of order $q$, which satisfies the following two properties:

$1.$ any $t$ of the first $k$ columns of $M$ are linearly independent, and

$2.$ any $s$ of the first $k$ columns of $M$, along with the last $t-s$ columns of $M$, are
linearly independent.

\noindent Then there exists a linear AOA$(s, t, k, q)$.
\end{construction}

Let $M$ be a $t$ by $k + t-s$ matrix over $\mathbb{F}_q$ satisfying the two properties in Construction \ref{linear AOA}. By linear algebra theory, there is a $t$ by $t$ invertible  matrix $P$ over $\mathbb{F}_q$ such that $PM$ is of the form
\[
\left ( \begin{array}{cc}
M_1 & 0 \\
M_2 & E_{t-s}\\
\end{array}
\right ),
\]
where $M_1$ is an $s$ by $k$ matrix, $M_2$ is a $t-s$ by $k$ matrix and $E_{t-s}$ is an identity matrix of order $t-s$. Since $P$ is invertible and $M$ has the two properties, the submatrix $\left ( \begin{array}{c}
M_1  \\
M_2 \\
\end{array}
\right )$ has the following properties:   (1') any $t$ columns  are linearly independent, and (2')
 any $s$  columns of $M_1$ are linearly independent. Conversely, if there is a $t$ by $k$ matrix  $\left ( \begin{array}{c}
M_1  \\
M_2 \\
\end{array} \right ) $ with  these two properties, then  the matrix $\left ( \begin{array}{cc}
M_1 & 0 \\
M_2 & E_{t-s}\\
\end{array}
\right )$  satisfies the properties in Construction \ref{linear AOA}. Therefore,  Construction \ref{linear AOA} can be rewritten as follows.

\begin{construction}
\label{linear AOA'}
 Suppose that $q$ is a prime power. Suppose there is a $t$ by $k$
matrix $M$, having entries from $\mathbb{F}_q$, which satisfies the following two properties:

$(1)$ any $t$  columns of $M$ are linearly independent, and

$(2)$ there are $s$ rows such that any $s$ columns of these $s$ rows are
linearly independent.

\noindent Then there exists a linear AOA$(s, t, k, q)$.
\end{construction}

A maximum distance separable code (MDS code) of length $k$ and size $v^t$ over an alphabet $X$ of size $v$ is a set of $v^t$ vectors (called codewords) in $X^k$, having the property that the hamming distance between any two codewords is at least $k-t+1$. It is well known that an OA$(t,k,v)$ is equivalent to an MDS code of length $k$ and size $v^t$ over an alphabet  of size $v$. It is also well known that  a linear OA$(t,k,q)$ defined over $\mathbb{F}_q$ is equivalent to a linear MDS code of length $k$ and dimension $t$ over $\mathbb{F}_q$. Also,  the dual code of a linear MDS code of length $k$ and dimension $t$ over $\mathbb{F}_q$ is a linear MDS code of length $k$ and dimension $k-t$ over $\mathbb{F}_q$. A linear code of length $k$  and dimension $t$ over $\mathbb{F}_q$ generated by the $t$ by $k$ matrix $M$ is a linear MDS code  if and only if  any $t$ columns of $M$ are linearly independent. There is a linear OA$(t,q+1,q)$ for any $t\leq q$.
For more information on linear MDS codes, we refer the reader to \cite{Roth2006}.

If a  $t$ by $k$ matrix $M$ over $\mathbb{F}_q$ satisfies the properties in Construction \ref{linear AOA'}, then there is a linear MDS code of length $k$ and dimension $t$ over $\mathbb{F}_q$ which contains a linear MDS subcode of length $k$ and dimension $s$ over $\mathbb{F}_q$. Conversely, if there is a linear MDS code of length $k$ and dimension $t$ over $\mathbb{F}_q$ which contains a linear MDS subcode of length $k$ and dimension $s$ over $\mathbb{F}_q$, there is a  $t$ by $k$ matrix $M$ over $\mathbb{F}_q$ satisfies the properties in Construction \ref{linear AOA'}. Therefore, a linear AOA can be characterized in terms of a specially linear MDS code.

\begin{theorem}
\label{linear AOA-MDS}
There is a linear AOA$(s,t,k,q)$ if and only if there is a linear MDS code of length $k$ and dimension $t$ over $\mathbb{F}_q$ which contains a linear MDS subcode of length $k$ and dimension $s$ over $\mathbb{F}_q$.
\end{theorem}

Stinson has proved that if there is a linear OA$(t-s,t,q)$ then there is a linear AOA$(s,t,t,q)$ \cite{Stinson2017}. We can give an alternative proof by Theorem \ref{linear AOA-MDS}.
When there is  a linear OA$(t-s,t,q)$, there is a linear MDS code of length $t$ and dimension $t-s$. Its dual code ${\cal C}$ is a linear MDS code  of length $t$ and dimension $s$. Since the trivial  linear MDS code $\mathbb{F}_q^t$ of length $t$ and dimension $t$ contains ${\cal C}$, there is a linear  AOA$(s,t,t,q)$ by Theorem \ref{linear AOA-MDS}.

\begin{theorem}
If there is a linear OA$(s,t,q)$, then there is a linear AOA$(s,t,t,q)$ and a linear AOA$(t-s,t,t,q)$.
\end{theorem}

\begin{theorem}
If there is a linear AOA$(s,t,k,q)$, then there is a linear AOA$(k-t,k-s,k,q)$.
\end{theorem}

\proof By assumption and Theorem \ref{linear AOA-MDS}, there is a a linear MDS code ${\cal C}$ of length $k$ and dimension $t$ over $\mathbb{F}_q$ which contains a linear MDS subcode ${\cal C}'$ of length $k$ and dimension $s$ over $\mathbb{F}_q$. The dual code of ${\cal C}'$ contains the dual code of ${\cal C}$.  Since the dual code of ${\cal C}'$ is a linear MDS code of length $k$ and dimension $k-s$ and the dual code of ${\cal C}'$ is a linear MDS code of length $k$ and dimension $k-t$, there is a linear MDS code of length $k$ and dimension $k-s$ containing a linear MDS subcode of length $k$ and dimension $k-t$.  The conclusion then follows from Theorem \ref{linear AOA-MDS}. \qed

\begin{theorem}
\label{AOA(1,k-1,k,q)}
Suppose that $q>2$ is a prime power and $k\geq 3$ is an integer.  There is a linear AOA$(1,k-1,k,q)$.
\end{theorem}

\proof Clearly, there is an element $\alpha \in \mathbb{F}_q\setminus \{0\}$ such that $\alpha +k-2\neq 0$. It is well known that the dual code ${\cal C}$ of the linear MDS code of length $k$ generated by $(1,1,\ldots, 1)$    is a linear MDS code of length $k$ and dimension $k-1$. Furthermore, ${\cal C}$ contains a codeword $(\alpha +k-2,-\alpha,-1,-1,\ldots, -1)$. Therefore, ${\cal C}$ contains a linear MDS subcode of length $k$ and dimension 1. The conclusion then  follows from Theorem \ref{linear AOA-MDS}. \qed

 \begin{theorem}
 \label{AOAt-sGeq2}
Suppose that  $q\geq 3$ is a prime power and that $s,t$ are  integers with $1\leq s<t\leq q$  and  $t-s\geq 2$.
There exists a linear AOA$(s,t,q+1,q)$.
\end{theorem}
\proof Denote $\mathbb{F}_q=\{a_1,a_2,\ldots,a_q\}$. Take a monic irreducible polynomial $h(x)$ of degree $t-s$ over $\mathbb{F}_q$ and define
 \[ M=\left(
    \begin{array}{cccccc}
      1 &1 &  \cdots & 1 & 0 \smallskip \\
      a_1&a_2 & \cdots & a_q & 0 \smallskip \\
      \vdots&\vdots & \ddots& \vdots & \vdots \smallskip \\
      a_1^{t-s-1} &a_2^{t-s-1} &  \cdots & a_q^{t-s-1}&0  \smallskip\\
      h(a_1) & h(a_2) &  \cdots & h(a_q) &0  \smallskip\\
      a_1h(a_1) & a_2h(a_2) & \cdots & a_qh(a_q) &0  \smallskip\\
     \vdots&\vdots &   \ddots& \vdots & \vdots \smallskip \\
    a_1^{s-1}h(a_1) & a_2^{s-1}h(a_2) &  \cdots & a_q^{s-1}h(a_q) &1  \smallskip\\
    \end{array}
  \right).\]
Note that for $s=1$, the last row of $M$ is   $(h(a_1), h(a_2),  \ldots, h(a_q),1)$.

  Since $h(x)$ is a monic irreducible polynomial  of degree $t-s$ over $\mathbb{F}_q$, each of $h(a_1),h(a_2),\ldots, h(a_q)$ is nonzero.
It is easy to see from Vandermoode determinant
 that any $t$ columns of $M$  are linearly independent and that the last $s$ rows satisfy the property (2) of Construction \ref{linear AOA'}.  Applying Construction \ref{linear AOA'} yields the result. \qed

For $t-s=1$, by Theorem \ref{AOA(t-1,t,k,v)}  an AOA$(t-1,t,q+1,q)$ implies an OA$(t,q+2,q)$. It is well known that there is a linear OA$(3,q+2,q)$ when $q\geq 4$ is a prime power of 2. However,  for other parameters $t,q$, no  OA$(t,q+2,q)$ has been found.

 \begin{theorem}
 \label{AOA(2,3,q+1,q)}
Suppose that  $q\geq 4$ is a prime power of $2$.
There exists a linear AOA$(2,3,q+1,q)$.
\end{theorem}
\proof Denote $\mathbb{F}_q=\{a_1,a_2,\ldots,a_q\}$ and define
 \[ M=\left(
    \begin{array}{cccccc}
      1 &1 &  \cdots & 1 & 0 \smallskip \\
      a_1&a_2 &  \cdots & a_q & 0 \smallskip \\
      a_1^{2} &a_2^{2} &  \cdots & a_q^{2}& 1 \smallskip\\
    \end{array}
  \right).\]

It is easy to see that
 the matrix $M$ satisfies the property (1) of
Construction \ref{linear AOA'} and the submatrix consisting of the first row and the last row also satisfies the property (2) of
Construction \ref{linear AOA'}.  Therefore, there is a linear AOA having the stated parameters.  \qed

\begin{theorem}
\label{linear AOA(1,3,q+2,q)}
\label{AOA(1,3,q+2,q)}
If $q\geq 4$ is a
power of $2$ then  there is a linear AOA$(1,3,q+2,q)$.
\end{theorem}

\proof It is easy to see that there is an element $\alpha\in \mathbb{F}_q\setminus \{0\}$, so that $\alpha\neq a+a^{-1}$ for any $a\in \mathbb{F}_q\setminus \{0\}$. It follows that the polynomial $h(x)=x^2+\alpha x+1$ is  irreducible over $\mathbb{F}_q$. Denote $\mathbb{F}_q=\{a_1,a_2,\ldots,a_q\}$ and define
\begin{center}
$M=\begin{pmatrix}
h(a_1) & h(a_2)  & \cdots & h(a_q) & 1  & 1 \smallskip \\
\alpha a_1 & \alpha a_2  & \cdots & \alpha a_q & 0 & 1  \smallskip\\
a_1^2 & a_2^2  & \cdots & a_q^2 & 1 & 0  \smallskip \\
\end{pmatrix}$.
\end{center}
The matrix $M$ satisfies the conditions of
Construction \ref{linear AOA'} where each coordinate of the first row is nonzero, and therefore it yields a linear AOA having the stated parameters.  \qed

\begin{theorem}
\label{linear AOA(1,q-1,q+2,q)}
If $q\ge 4$ is a power of $2$ then there exists a linear AOA$(1,q-1,q+2,q)$ and AOA$(3,q-1,q+2,q)$.
\end{theorem}

\begin{proof}
Denote $\mathbb{F}_q=\{0,a_1,a_2,\ldots,a_{q-1}\}$ and take a primitive element $\alpha$. Define
\[
M=\left(
    \begin{array}{c}
      m_1 \\
      m_2 \\
      \vdots \\
      m_{q-2} \\
      m_{q-1} \\
    \end{array}
  \right)
=\left(
  \begin{array}{cccccccc}
    1 & a_1 & a_1^2 & 1 & 0 &\cdots& 0 & 0 \\
    1 & a_2 & a_2^2 & 0 & 1 & \cdots&0 & 0\\
    \vdots & \vdots & \vdots &\vdots& \vdots & \ddots & \vdots & \vdots \\
    1 & a_{q-2} & a_{q-2}^2 & 0 & 0 & \cdots& 1 & 0 \\
    1 & a_{q-1} & a_{q-1}^2 & 0 & 0 & \cdots& 0 & 1 \\
  \end{array}
\right).
\]
It is routine to check that any $q-1$ columns of $M$ are linearly independent. We choose the row vector $\mathbf{v}=\alpha m_1+m_2+\cdots+m_{q-1}=(\alpha,(\alpha-1)a_1,(\alpha-1)a_1^2,\alpha,1,\cdots,1)$, each coordinate of this row is nonzero.  Applying Construction \ref{linear AOA'} yields a linear AOA$(1,q-1,q+2,q)$.

For the case of AOA$(3,q-1,q+2,q)$, we choose $\mathbf{v_1}=m_1+m_2+\cdots+m_{q-1}=(1,0,0,1,\cdots,1)$, $\mathbf{v_2}=a_1^{q-3}m_1+a_2^{q-3}m_2+\cdots+a_{q-1}^{q-3}m_{q-1}=(0,0,1,a_1^{q-3},\cdots,a_{q-1}^{q-3})$ and $\mathbf{v_3}=a_1^{q-2}m_1+a_2^{q-2}m_2+\cdots+a_{q-1}^{q-2}m_{q-1}=(0,1,0,a_1^{q-2},\cdots,a_{q-1}^{q-3})$. It is obvious that any three columns of the submatrix consisting of $\mathbf{v_1},\mathbf{v_2},\mathbf{v_3}$ are linearly independent.  Applying Construction \ref{linear AOA'} yields a linear AOA$(1,q-1,q+2,q)$.
 \qed
\end{proof}


\section{Conclusion}

Augmented orthogonal arrays are equivalent to ideal ramp threshold schemes. The existence of AOAs is worth studying. In this paper, we showed that there is an AOA$(s,t,k,v)$ if and only if there is an OA$(t,k,v)$ which can be partitioned into $v^{t-s}$ subarrays, each being an OA$(s,k,v)$. We also  rewrote  Stinson's construction for linear AOAs and proved that there is a linear AOA$(s,t,k,q)$ if and only if there is a linear MDS code of length $k$ and dimension $t$ over $\mathbb{F}_q$ which contains a linear MDS subcode of length $k$ and dimension $s$ over $\mathbb{F}_q$. Many infinite classes of AOAs over alphabets of non-prime power orders and linear AOAs were obtained in Sections 2 and 3.

From the three problems posed by Stinson, it is more interesting to identify the parameters  for which there exists an AOA$(s, t , k, q)$ but there does not exist an OA$(t , k+t-s,q)$. We use Bush bound and Theorem \ref{linear bound} to identify some parameters.

\begin{theorem} [\cite{Bush1952a}]$(${\rm Bush Bound}$)$
\label{Bush Bound} If there is an OA$(t, k, v)$, then
\[
k\leq \left \{ \begin{array}{ll}
v + t -1  & {\rm if}\ t = 2,\ {\rm or\ if}\ v\ {\rm is\ even\ and}\ 3\leq t \leq v\\
v+t-2 & {\rm if}\ v\ {\rm is\ odd\ and}\ 3\leq t\leq v\\
t + 1 & {\rm if}\  t\geq v.
\end{array}
\right .
\]
\end{theorem}

There have been some relatively minor improvements to these general bounds over the years. On the other hand, the linear case has received considerably more attention and much more is known in this case.

The following is known as the Main Conjecture for linear MDS codes. It is attributed to Segre (1955).

\noindent {\bf Conjecture 1} (Main Conjecture). Suppose $q$ is a prime power. Let $M (t, q)$ denote the maximum value of $k$ such that there exists
a linear MDS code of length $k$ and dimension $t$ over $\mathbb{F}_q$. If $2 \leq t < q$, then
\[
M(t,q)=\left \{
\begin{array}{ll}
q+2 & {\rm if}\ $q$\ {\rm is\ a\ power\ of}\ 2\ {\rm and}\ t\in \{3,q-1\} \\
 q+1 & {\rm otherwise}.
\end{array}
\right .
\]
If $t\geq q$, then $M (t , q) = t + 1$.

The {\bf Main Conjecture} has been shown to be true in many parameter situations, including all the cases where $q$ is prime. This is a famous result of Simeon Ball \cite{Ball2012} proven in 2012.

The following theorem summarizes some of the known results. These and other related results are surveyed in \cite{Huntemann2012} .

\begin{theorem} \label{linear bound} Suppose that $q = p^j$ where $p$ is prime, and suppose $2\leq t < q$. Then the {\bf Main Conjecture} is true in the following
cases:

$1.$ $q$ is prime $($for all relevant $t)$ \\
\indent $2.$ $q\leq 27$ $($for all relevant $t)$\\
\indent $3.$ $t\leq 5$ or $t\geq q-3$\\
\indent $4.$ $t \leq p$.
\end{theorem}

Based on Bush bound and Theorem \ref{linear bound}, we  check AOAs and linear AOAs in Section 2 and Section 3 to identify the following parameters for which there exists an (linear) AOA$(s, t , k, v)$ but there does not exist an (linear) OA$(t , k+t-s,v)$, where $q$ is a prime power.

\begin{center}
\begin{tabular}{|c|c|c|c|}
\hline
AOA$(s,t,k,v)$ & Conditions &  OA$(t,k+t-s,v)$ & Sources \\ \hline
linear AOA$(1, t , q, q)$ &  $q$ is  an odd prime power, &  no  & \cite{Stinson2017}\\
&  $3\leq t\leq q$  & & \\ \hline
linear AOA$(s, q+1, q+1, q)$ &  $s\leq q-1$ &  no  & \cite{Stinson2017}\\ \hline
AOA$(1,k-1,k,v)$ & $v\not\equiv 2\pmod 4, k>v\geq 3$ & no  & Theorems \ref{AOA(1,k-1,k,v)}, \ref{Bush Bound} \\ \hline
AOA$(1,k-1,k,v)$ & $k\equiv 0\pmod 2, k>v\geq 2$ & no & Theorems \ref{zero-sum}, \ref{Bush Bound} \\ \hline
linear AOA$(1,k-1,k,q)$ & $q>2, k\geq q$ & no  & Theorems \ref{AOA(1,k-1,k,q)}, \ref{Bush Bound} \\ \hline
linear AOA$(s,t,q+1,q)$ & $s\in \{1,2\}$, $s+2\leq t\leq q$ & no   & Theorems \ref{AOAt-sGeq2}, \ref{Bush Bound} \\ \hline
linear AOA$(1,3,q+2,q)$ & $q=2^{\alpha}\geq 4$ & no  & Theorems \ref{linear AOA(1,3,q+2,q)}, \ref{Bush Bound} \\ \hline
linear AOA$(1,q-1,q+2,q)$ & $q=2^{\alpha}\geq 4$ & no  & Theorems \ref{linear AOA(1,q-1,q+2,q)}, \ref{Bush Bound} \\ \hline
linear AOA$(3,q-1,q+2,q)$ & $q=2^{\alpha}>4$ & no linear OA & Theorems \ref{linear AOA(1,q-1,q+2,q)}, \ref{linear bound} \\ \hline
\end{tabular}
\end{center}
By Theorems \ref{AOAt-sGeq2}, \ref{linear bound}, we can find many other parameters $s,t,q$ for which there is a linear AOA$(s,t,q+1,q)$ and there does not exist  a linear OA$(t,q+1+t-s,q)$.

\bigskip
\noindent {\bf Acknowledgements}

The authors would like to thank Prof. D. R. Stinson  who suggested us  identifying the parameters  for which there exists an AOA$(s, t , k, q)$ but there does not exist an OA$(t , k+t-s,q)$.

\end{document}